\documentclass[a4paper]{amsart}

\pdfoutput=1

\title{Minor exclusion and planar rigidity}
\author{Joseph Paul MacManus}
\date{8th October 2026}


\address{School of Mathematics,  University of Bristol, Bristol, BS8 1UG, UK, and the Heilbronn Institute for Mathematical Research, Bristol, UK}
\email{joseph.macmanus@bristol.ac.uk}

\usepackage{amsfonts}
\usepackage{amsmath}
\usepackage{amsthm}
\usepackage{amssymb}
\usepackage{stmaryrd}
\usepackage{mathrsfs}

\usepackage{graphicx}
\usepackage[dvipsnames]{xcolor}
\usepackage{tikz}
\usepackage{fancyref}
\usepackage{todonotes}

\usepackage{thmtools, thm-restate}

\usepackage[abbrev]{amsrefs}

\usepackage[colorlinks]{hyperref} 
\hypersetup{
    linkcolor=black,
    citecolor=blue,
}

\makeatletter
\newcommand{\myitem}[1]{%
\item[#1]\protected@edef\@currentlabel{#1}%
}
\makeatother

\newtheorem{theorem}{Theorem}[section]
\newtheorem{alphtheorem}{Theorem}

\newtheorem{alphcor}[alphtheorem]{Corollary}
\newtheorem*{theorem*}{Theorem}

\newtheorem{proposition}[theorem]{Proposition}
\newtheorem{lemma}[theorem]{Lemma}

\newtheorem{conjecture}[theorem]{Conjecture}

\theoremstyle{definition}
\newtheorem{definition}[theorem]{Definition}

\newtheorem{remark}[theorem]{Remark}

\begin{document}

\begin{abstract}
    We prove that every connected, locally finite, quasi-transitive graph which is quasi-isometric to a graph excluding a countable clique as a minor is quasi-isometric to a planar graph. Consequently, every finitely generated group with such a quasi-isometry is virtually planar, meaning that some finite-index subgroup admits a planar Cayley graph. 
    
    We also show that a connected, locally finite, quasi-transitive graph admitting a regular map into a bounded-degree graph excluding some finite minor admits a regular map into a bounded-degree planar graph. If the source is accessible, the target need only exclude a countable clique. Combined with work of Mackay, Spriano, and the present author, this shows that a finitely presented group admits a regular map into a bounded-degree graph excluding a countable clique if and only if it is virtually planar.
\end{abstract}

\maketitle




\section{Introduction}

It is known that a finitely generated group is \emph{virtually planar}---meaning that some finite-index subgroup admits a planar Cayley graph---as soon as its quasi-isometry class contains a planar graph \cite{macmanus2023accessibility}. This theorem belongs to a broader history of rigidity results for groups whose geometry can be modelled by planar spaces. A classical instance is the characterisation of virtual surface groups as precisely the finitely generated groups quasi-isometric to complete Riemannian planes. Originating in work of Mess, this characterisation was completed using the convergence group theorem of Tukia, Gabai, and Casson--Jungreis \cites{mess1988seifert,tukia1988homeomorphic,gabai1992convergence,casson1994convergence}. Alternative proofs, extensions, and related rigidity results appear in \cites{maillot2001quasi,bowditch2004planar,papasoglu2005quasi,macmanus2023accessibility,macmanus2024note,macmanus2025fat,macmanus2025vertex,mackay2026almost}. 
From a complementary, graph-theoretic direction, Davies \cite{davies2025string} and Georgakopoulos–Molinari \cite{georgakopoulos2026quasi} have recently established combinatorial conditions under which a graph is quasi-isometric to a planar graph. Together with \cite{macmanus2023accessibility}, these results provide further characterisations of virtually planar groups.

The purpose of this paper is to explore the link between this rigidity and minor exclusion. 
Within the realm of infinite graphs, the interaction between minor exclusion and transitivity was first studied by Thomassen \cite{Tho92}, who showed that every locally finite, one-ended, vertex-transitive graph is either planar or contains $K_{\omega}$ as a minor. Minor exclusion in Cayley graphs was previously studied by Ostrovskii–Rosenthal \cite{OR15} and Khukhro \cite{Khu23}, who proved that a finitely generated group is virtually free if and only if every one of its Cayley graphs excludes some finite minor.
More recently, Esperet, Giocanti, and Legrand-Duchesne proved a canonical structure theorem which applies to all locally finite, quasi-transitive graphs which exclude a countable clique \cite{esperet2024structure}. In particular, such a graph is quasi-isometric to a planar graph \cite{Esperet_Giocanti_2024}, and thus if some Cayley graph of a finitely generated group $G$ excludes a countable clique then $G$ is virtually planar. 

In the spirit of the above, we now state our main theorem, which asserts that it is enough for some arbitrary graph in the quasi-isometry class of our quasi-transitive input to exclude a countable clique---with no transitivity assumption on this graph---to force large-scale planarity.

\begin{restatable}{alphtheorem}{main}\label{thm:main}
    Suppose a connected, locally finite, quasi-transitive graph $\Gamma$
    is quasi-isometric to a graph which excludes a countable clique.
    Then $\Gamma$ is quasi-isometric to a planar graph.
\end{restatable}

We remark that one can further upgrade
the planar graph in Theorem~\ref{thm:main} to be a planar Cayley graph;
see \cite{macmanus2024note} for details. 
If we combine Theorem~\ref{thm:main} with
\cite[Cor.~D]{macmanus2023accessibility}, we deduce the following corollary.

\begin{alphcor}\label{cor:planar-rigidity}
    Suppose a finitely generated group $G$ is quasi-isometric to a graph
    which excludes a countable clique. Then $G$ is virtually planar.
\end{alphcor}

This directly improves upon \cite[Cor.~B]{macmanus2025fat}, which asserts that a finitely presented group which is quasi-isometric to a graph excluding some finite minor is virtually planar.

Our methods actually apply much more generally, allowing us to work within the paradigm of \emph{regular maps} rather than quasi-isometries. Recall that a map between the vertex sets of connected, bounded-degree graphs is \emph{regular} if it is Lipschitz and bounded-to-one. Introduced by Benjamini--Schramm in \cite{benjamini1996harmonic},
these maps generalise the notion of a coarse embedding, in essence by allowing arbitrarily large `folds'.
Many large-scale invariants, such as asymptotic dimension and separation profile, are monotone under regular maps.

Benjamini and Schramm define an \emph{almost planar graph} to mean a bounded-degree graph which admits a regular map to some bounded-degree planar graph. Equivalently, an almost planar graph is one which can be drawn in the plane such that every edge crosses a bounded number of other edges \cite{benjamini1996harmonic}. 
Applying the techniques used to prove Theorem~\ref{thm:main} to regular maps, we prove the following characterisation of almost planar quasi-transitive graphs.

\begin{restatable}{alphtheorem}{regfinite}\label{thm:reg-finite}
    Suppose a connected, locally finite, quasi-transitive graph $\Gamma$
    admits a regular map to a bounded-degree graph which excludes some finite minor. Then $\Gamma$ is almost planar. 
\end{restatable}

The reason we require the target graph in Theorem~\ref{thm:reg-finite} to exclude some \emph{finite} minor, rather than a countable clique, comes down to accessibility. More precisely, the issue is that there is currently no known analogue of \cite[Thm.~A]{macmanus2023accessibility} for almost planar graphs. 
If we include accessibility in our assumptions, we can prove the following stronger result.

\begin{restatable}{alphtheorem}{regctble}\label{thm:reg-ctble}
    Suppose a connected, locally finite, quasi-transitive, accessible graph $\Gamma$
    admits a regular map to a bounded-degree graph which excludes a countable clique. Then $\Gamma$ is almost planar. 
\end{restatable}

The \emph{$k$-planar conjecture}, which was first posed by Georgakopoulos--Papasoglu \cite[Prob.~9.4]{georgakopoulos2025graph}, asserts that every almost planar, locally finite, quasi-transitive graph is quasi-isometric to a planar graph. 
This conjecture was recently settled for `coarsely simply connected' graphs by Mackay, Spriano, and the present author \cite{mackay2026almost}, though it remains open in full generality. In particular, it is known that almost planar finitely presented groups are virtually planar. Since finitely presented groups are accessible \cite{dunwoody1985accessibility}, our Theorem~\ref{thm:reg-ctble} yields the following strong geometric characterisation of virtual planarity amongst finitely presented groups.

\begin{alphcor}\label{cor:planar-rigidity-regularmaps}
     A finitely presented group admits a regular map to a bounded-degree graph which excludes a countable clique if and only if it is virtually planar. 
\end{alphcor}

More generally, our results show that the $k$-planar conjecture is equivalent to the following, \emph{a priori} stronger, conjecture.

\begin{conjecture}\label{conj:k-planar-strong}
    Every connected, locally finite, quasi-transitive graph admitting a regular map into a bounded-degree graph excluding a countable clique is quasi-isometric to a planar graph.
\end{conjecture}

Indeed, assume the $k$-planar conjecture, and let $\Gamma$ satisfy the hypotheses of Conjecture~\ref{conj:k-planar-strong}. If $\Gamma$ were inaccessible, Lemma~\ref{prop:inaccessible-almost-planar-part} would produce an inaccessible, almost planar, locally finite, quasi-transitive graph $P$. The $k$-planar conjecture would make $P$ quasi-isometric to a planar graph, contradicting \cite[Thm.~A]{macmanus2023accessibility}. Thus $\Gamma$ is accessible. Theorem~\ref{thm:reg-ctble} makes $\Gamma$ almost planar, and the $k$-planar conjecture then gives quasi-planarity. Conversely, every almost planar, connected, locally finite, quasi-transitive graph admits a regular map into a bounded-degree planar graph by definition, so Conjecture~\ref{conj:k-planar-strong} implies the $k$-planar conjecture.

\subsection{Outline}

We briefly outline the proof of Theorem~\ref{thm:main}. The starting point is the structure theorem of Diestel--Thomas \cite{diestel1999excluding}, which decomposes the target as a tree of torsos, each planar apart from a finite deleted set and at most one `vortex'. Our first step is to show that a locally finite, quasi-transitive graph which is quasi-isometric to a connected bounded-degree subgraph of such a torso is quasi-isometric to a planar graph. In other words, we can `escape the vortex'; see Theorem~\ref{thm:escaping-vortex} below. This combines Grohe's theorem on local tree-width \cite{Grohe03} with a pointed-limit argument.

The principal difficulty is then that our quasi-isometry need not interact in any controlled way with the Diestel--Thomas decomposition. In the one-ended case, we show that the geometry of our graph localises to a single torso. The preceding result then settles this case; see Theorem~\ref{thm:one-ended-rigidity} below.

The remaining issue is accessibility. If our quasi-transitive graph $\Gamma$ were inaccessible, the Dicks--Dunwoody theory of cuts \cite{dicks1989groups} would allow us to find a quasi-transitive part $P \subseteq \Gamma$ which retains an ultra-thick end and whose distinct ends require many vertices to separate them. Choosing this separation threshold sufficiently large forces the geometry of $P$ to localise to a single target torso, making $P$ quasi-isometric to a planar graph. However, its ultra-thick end makes $P$ inaccessible, contradicting \cite[Thm.~A]{macmanus2023accessibility}. Thus $\Gamma$ is accessible, and we conclude by applying the one-ended case and assembling the resulting planar graphs.

To prove Theorems~\ref{thm:reg-finite} and \ref{thm:reg-ctble}, the key observation is that the same localisation and limiting arguments also apply to regular maps. The proofs of all three theorems therefore essentially proceed in parallel.

\subsection*{Acknowledgements}

We thank Louis Esperet, Ugo Giocanti, Bobby Miraftab, and Panos Papasoglu for comments on an earlier draft of this paper. 
This work was supported by the Additional Funding Programme for Mathematical Sciences, delivered by EPSRC (EP/V521917/1) and the Heilbronn Institute for Mathematical Research.


\section{Preliminaries}\label{sec:preliminaries}

We collect the facts about coarse geometry and structural graph theory which will be used throughout this paper.

\subsection{Graphs}

All graphs are simple and undirected, and may be infinite.
Distances in a connected graph are measured in its path metric, and all neighbourhoods are closed neighbourhoods of vertex sets.
A graph is \emph{quasi-transitive} if its automorphism group has finitely many vertex orbits; in particular, a locally finite quasi-transitive graph has bounded degree.
We write $F\prec\Gamma$ when $F$ is a minor of $\Gamma$, and say that $\Gamma$ \emph{excludes} $F$ when $F\nprec\Gamma$.
We denote the finite clique of order $n$ by $K_n$ and the countably infinite clique by $K_\omega$. Given $S \subset V(\Gamma)$, we write $\Gamma[S]$ for the subgraph induced by $S$. We will usually abuse notation and write $f:\Gamma\to\Lambda$ for a map between vertex sets.

If $\Gamma$ has bounded degree, we let $v_\Gamma(r)$ denote the maximum cardinality of a ball of radius $r$ in $\Gamma$.

\subsection{Decompositions}\label{subsec:decompositions}

Let $\Gamma$ be a graph.
A \emph{tree-decomposition} of $\Gamma$ is a pair $(T,\beta)$, where $T$ is a tree and $\beta:V(T)\to\mathcal{P}(V(\Gamma))$ satisfies the following conditions:
\begin{enumerate}
\item[(T1)] We have $V(\Gamma)=\bigcup_{t\in V(T)}\beta(t)$.
\item[(T2)] For every edge $xy\in E(\Gamma)$, there exists $t\in V(T)$ such that $x,y\in\beta(t)$.
\item[(T3)] For every $x\in V(\Gamma)$, the set $\{t\in V(T):x\in\beta(t)\}$ induces a connected subtree of $T$.
\end{enumerate}
The sets $\beta(t)$ are called the \emph{bags} of the decomposition, and the subgraphs $\Gamma[\beta(t)]$ induced by the bags are called \emph{parts}.
Its \emph{width} is $\operatorname{width}(T,\beta):=\sup_{t\in V(T)}|\beta(t)|-1$.
The \emph{tree-width} of $\Gamma$, denoted $\operatorname{tw}(\Gamma)$, is the infimum of these widths over all tree-decompositions of $\Gamma$.
A tree-decomposition whose underlying tree is a path is called a \emph{path-decomposition}.
For every edge $st\in E(T)$, the set $\beta(s)\cap\beta(t)$ is called an \emph{adhesion set}.
The decomposition has \emph{bounded adhesion} if the cardinalities of its adhesion sets are finite and uniformly bounded.
We will also use the following weaker notion.

\begin{definition}[Finite adhesion]\label{def:finite-adhesion}
A tree-decomposition $(T,\beta)$ has \emph{finite adhesion} if every adhesion set is finite and, for every ray $t_1,t_2,\ldots$ in $T$, we have $\liminf_{i\to\infty}|\beta(t_i)\cap\beta(t_{i+1})|<\infty$.
\end{definition}

For $t\in V(T)$, the \emph{torso} at $t$ is the graph with vertex set $\beta(t)$ obtained from $\Gamma[\beta(t)]$ by making $\beta(t)\cap\beta(s)$ a clique for every neighbour $s$ of $t$.
We denote this graph by $\Gamma[\![\beta(t)]\!]$.
Every edge $st\in E(T)$ induces a separation of $\Gamma$.
The two components of $T-st$ determine two unions of bags whose intersection is $\beta(s)\cap\beta(t)$.
We call these unions the \emph{sides} of the separation, and call the sets obtained by removing the adhesion set its \emph{strict sides}.
If a group $A$ acts on $\Gamma$ and on $T$ with $\beta(at)=a\beta(t)$ for all $a\in A$ and $t\in V(T)$, the decomposition is \emph{$A$-invariant}.
We write $A_t$ for the stabiliser of $t$ in this tree action.
A decomposition invariant under the full automorphism group is called \emph{canonical}.
We will also need the following generalisation of a path-decomposition.

\begin{definition}[Linear decomposition]\label{def:linear-decomposition}
A \emph{linear decomposition} of a graph $\Gamma$ is a pair $(L,\beta)$, where $(L,\leq)$ is a totally ordered set and $\beta:L\to\mathcal{P}(V(\Gamma))$ satisfies the following conditions:
\begin{enumerate}
\item[(L1)] We have $V(\Gamma)=\bigcup_{t\in L}\beta(t)$.
\item[(L2)] For every edge $xy\in E(\Gamma)$, there exists $t\in L$ such that $x,y\in\beta(t)$.
\item[(L3)] If $r\leq s\leq t$, then $\beta(r)\cap\beta(t)\subseteq\beta(s)$.
\end{enumerate}
The width of this decomposition is $\sup_{t\in L}|\beta(t)|-1$.
\end{definition}

\subsection{Coarse geometry}\label{subsec:coarse-geometry}

A map $f:\Gamma\to \Lambda$ between connected graphs is a \emph{$\lambda$-quasi-isometric embedding}, for $\lambda\geq1$, if
\begin{equation}\label{eq:quasi-isometry-inequalities}
\lambda^{-1}d_\Gamma(x,y)-\lambda\leq d_\Lambda(f(x),f(y))\leq\lambda d_\Gamma(x,y)+\lambda
\end{equation}
for all $x,y\in V(\Gamma)$.
We say $f$ is ($\lambda$-)\emph{coarsely surjective} if its image is \emph{$\lambda$-coarsely dense}, meaning that every vertex of $\Lambda$ lies within distance $\lambda$ of $f(V(\Gamma))$. If $f$ is a $\lambda$-coarsely surjective $\lambda$-quasi-isometric embedding,  then $f$ is called a \emph{$\lambda$-quasi-isometry}.
We write $\Gamma\simeq_{\mathrm{QI}}\Lambda$ when $\Gamma$ and $\Lambda$ are quasi-isometric.
A connected graph is \emph{quasi-planar} if it is quasi-isometric to a planar graph, and is a \emph{quasi-tree} if it is quasi-isometric to a tree.

\begin{definition}[Regular map]\label{def:regular-map}
For an integer $\kappa\geq1$, a map $f:\Gamma\to\Lambda$ between connected, bounded-degree\footnote{The convention that regular maps must be between bounded-degree graphs is not strictly necessary but is entirely typical, as removing this convention requires using a more nuanced definition of a regular map; see e.g. \cite[Sec.~5.1]{hume2020poincare}. Indeed, if we drop the bounded-degree hypothesis na\"ively, then all connected graphs map regularly to a star graph.} 
graphs is \emph{$\kappa$-regular} if $d_\Lambda(f(x),f(y))\leq\kappa d_\Gamma(x,y)$ for all $x,y\in V(\Gamma)$ and $|f^{-1}(z)|\leq\kappa$ for every $z\in V(\Lambda)$.
A map is \emph{regular} if it is $\kappa$-regular for some $\kappa$.
\end{definition}

Compositions of regular maps are regular, and every quasi-isometric embedding between bounded-degree graphs is regular.
We will frequently use the following observation to make a regular map coarsely surjective.

\begin{lemma}\label{lem:regular-connected-image}
Let $f:\Gamma\to\Pi$ be a $\kappa$-regular map between connected graphs, where $\Pi$ has bounded degree.
There is a connected subgraph $H\subseteq\Pi$ containing $f(V(\Gamma))$ such that $f:\Gamma\to H$ is $\kappa$-regular and coarsely surjective.
\end{lemma}

\begin{proof}
For every edge $xy\in E(\Gamma)$, choose a path $P_{xy}$ in $\Pi$ of length at most $\kappa$ joining $f(x)$ to $f(y)$, allowing a constant path when these vertices coincide.
Let $H$ be the union of these paths.
Then $H$ is connected, $d_H(f(x),f(y))\leq \kappa d_\Gamma(x,y)$ for all $x,y\in V(\Gamma)$, and the fibres of $f$ are unchanged.
Every vertex of $H$ lies within distance $\kappa$ of the image, so $f:\Gamma\to H$ is coarsely surjective.
\end{proof}

A graph is \emph{$k$-planar} if it can be drawn in the plane so that every edge crosses at most $k$ other edges.
It is \emph{almost planar} if it is $k$-planar for some finite $k$.
The following characterisation is due to Benjamini--Schramm.

\begin{theorem}[{\cite[Thm.~1.8]{benjamini1996harmonic}}]\label{thm:almost-planar-regular}
A connected bounded-degree graph is almost planar if and only if it admits a regular map into a connected bounded-degree planar graph.
\end{theorem}

Consequently, if a connected bounded-degree graph admits a regular map into an almost planar bounded-degree graph, then it is almost planar.
In particular, almost planarity is invariant under quasi-isometry among connected bounded-degree graphs.
The following property bounds the sizes of finite components cut off by small vertex sets.

\begin{definition}\label{def:ddagger}
A connected graph $\Gamma$ satisfies property $(\ddagger)$ if, for every $a\geq0$, there exists $m=m(a)$ such that every finite component of $\Gamma-S$ has at most $m$ vertices whenever $S\subseteq V(\Gamma)$ has cardinality at most $a$.
\end{definition}

\begin{lemma}[{\cite[Prop.~3.8]{mackay2026almost}}]\label{lem:quasi-transitive-ddagger}
Every infinite, connected, locally finite, quasi-transitive graph which is not two-ended satisfies $(\ddagger)$.
\end{lemma}

\begin{lemma}[{\cite[Lem.~3.10]{mackay2026almost}}]\label{lem:regular-ddagger}
Let $f:\Gamma\to H$ be a coarsely surjective regular map between connected bounded-degree graphs.
If $\Gamma$ satisfies $(\ddagger)$, then $H$ satisfies $(\ddagger)$.
\end{lemma}

For quasi-isometric embeddings, we will also need the following reduction, which does not require the target to have bounded degree.

\begin{lemma}[{\cite[Prop.~1.11]{macmanus2023accessibility}}]\label{lem:bounded-degree-qi-image}
Let $f:\Gamma\to\Pi$ be a quasi-isometric embedding between connected graphs, where $\Gamma$ has bounded degree.
Then $\Pi$ contains a connected bounded-degree subgraph which is quasi-isometric to $\Gamma$.
\end{lemma}

\subsection{Ends and accessibility}\label{subsec:ends-accessibility}

Let $\Gamma$ be a connected locally finite graph.
Two rays in $\Gamma$ are \emph{end-equivalent} if, after deleting any finite vertex set, tails of the two rays lie in the same component.
The resulting equivalence classes are the \emph{ends} of $\Gamma$, and their set is denoted by $\Omega(\Gamma)$.
If $S\subseteq V(\Gamma)$ is finite and $C$ is a component of $\Gamma-S$, we say that an end $\omega$ \emph{lies in} $C$ if every ray representing $\omega$ has a tail in $C$.
For distinct ends $\omega,\xi\in\Omega(\Gamma)$, let $\operatorname{vs}(\omega,\xi)$ be the minimum cardinality of a vertex set separating them, and let $\operatorname{es}(\omega,\xi)$ be the minimum cardinality of an edge set separating them.
We set
\begin{equation}\label{eq:end-separation-infima}
\operatorname{vs}(\Gamma):=\inf_{\omega\neq\xi}\operatorname{vs}(\omega,\xi),
\qquad
\operatorname{es}(\Gamma):=\inf_{\omega\neq\xi}\operatorname{es}(\omega,\xi),
\end{equation}
with both quantities taken to be infinite when $\Gamma$ has at most one end.
If $\Gamma$ has maximum degree at most $\Delta$, then $\operatorname{vs}(\Gamma)\leq\operatorname{es}(\Gamma)\leq\Delta\operatorname{vs}(\Gamma)$.
Following Thomassen--Woess \cite{thomassen1993vertex}, the graph $\Gamma$ is called \emph{accessible} if there is a uniform finite bound on the number of edges needed to separate any two distinct ends.
For bounded-degree graphs this is equivalent to the analogous condition using vertex separators. Accessibility is invariant under quasi-isometry among connected bounded-degree graphs.

An end is \emph{thick} if it contains infinitely many pairwise vertex-disjoint rays.
An end $\omega$ is \emph{ultra-thick} if, for every $n\geq1$, there is an end $\xi\neq\omega$ such that $\operatorname{vs}(\omega,\xi)>n$.
Every ultra-thick end is thick.
By \cite[Thm.~5.4]{thomassen1993vertex}, a connected, locally finite, quasi-transitive graph is inaccessible if and only if it has an ultra-thick end; indeed, every such inaccessible graph has uncountably many ultra-thick ends \cite[Thm.~1]{moller1996accessibility}.
Quasi-isometries between connected bounded-degree graphs induce bijections on ends preserving thickness and ultra-thickness.
For regular maps, we will use the following weaker conclusions.

\begin{lemma}\label{lem:regular-map-on-ends}
Let $f:\Gamma\to\Lambda$ be a regular map between connected bounded-degree graphs. Then $f$ induces a map $f_*:\Omega(\Gamma)\to\Omega(\Lambda)$ which sends thick ends to thick ends. If $f$ is coarsely surjective, then $f_*$ is surjective.
\end{lemma}

\begin{proof}
Choose $\kappa$ such that $f$ is $\kappa$-regular, and extend $f$ over each edge by a path of length at most $\kappa$. This gives a proper continuous map, which induces a map $f_*:\Omega(\Gamma)\to\Omega(\Lambda)$ independent of the chosen edge paths; see e.g. \cite[Prop.~I.9.11]{baues2001infinite}. More generally, $f$ extends to a continuous map between the Freudenthal compactifications of $\Gamma$ and $\Lambda$. 
If $f$ is coarsely surjective, every end of $\Lambda$ lies in the closure of $f(V(\Gamma))$ in its Freudenthal compactification. The continuous extension of $f$ has compact image, so every such end has a preimage, necessarily in $\Omega(\Gamma)$. Thus $f_*$ is surjective.

Now let $\omega$ be thick and put $C:=\kappa v_\Lambda(\kappa)$. Routing pairwise disjoint rays in $\omega$ through $f$ and deleting loops gives rays in $f_*(\omega)$, with each target vertex lying on at most $C$ of them. Indeed, any source ray whose routed image meets $z$ meets $f^{-1}(B_\Lambda(z,\kappa))$. For every positive integer $n$, take $nC$ such image rays. In their union, every vertex set separating their initial vertices from infinity has size at least $n$, so Menger's theorem and local finiteness give $n$ pairwise disjoint rays, all belonging to $f_*(\omega)$. Halin's theorem that end degree is attained \cite[p.~241]{diestel2025graph} now implies that $f_*(\omega)$ is thick.
\end{proof}

This also recovers \cite[Lem.~3.9]{mackay2026almost}, which states that a coarsely surjective regular map between infinite connected bounded-degree graphs takes a one-ended source to a one-ended target.
We also need quantitative control of separators under regular maps.

\begin{lemma}\label{lem:regular-end-separators}
Let $f:\Gamma\to\Lambda$ be a regular map between connected bounded-degree graphs.
For every $a\geq0$, there exists $b=b(a)$ such that, whenever $\omega,\xi\in\Omega(\Gamma)$ have distinct image ends satisfying $\operatorname{vs}(f_*(\omega),f_*(\xi))\leq a$, we have $\operatorname{vs}(\omega,\xi)\leq b$.
\end{lemma}

\begin{proof}
Choose $\kappa$ such that $f$ is $\kappa$-regular.
Choose $S\subseteq V(\Lambda)$ separating the image ends, with $|S|\leq a$, and put $S':=f^{-1}(B_\Lambda(S,\kappa))$.
We have $|S'|\leq \kappa v_\Lambda(\kappa)a$.
Every path in $\Gamma-S'$ maps, after replacing its edges by paths of length at most $\kappa$, to a walk in $\Lambda-S$.
Consequently, $S'$ separates $\omega$ from $\xi$, and we may take $b=\kappa v_\Lambda(\kappa)a$.
\end{proof}

Let $(T,\beta)$ be a tree-decomposition of a connected locally finite graph with finite adhesion.
Every end orients each edge of $T$ towards the side of its induced separation containing a tail of that end.
We say that the end is \emph{carried by} a vertex $t\in V(T)$ if $t$ is a sink for this orientation.
The following observation extends \cite[Prop.~4.8(ii)]{hamann2022stallings} to the finite-adhesion convention used here.

\begin{lemma}\label{lem:thick-end-carried}
Let $\Gamma$ be a connected locally finite graph with a tree-decomposition $(T,\beta)$ of finite adhesion.
Every thick end of $\Gamma$ is carried by a vertex of $T$.
\end{lemma}

\begin{proof}
Let $\omega$ be a thick end and orient $T$ towards $\omega$.
If this orientation has no sink, successively following outgoing edges produces a directed ray $t_0,t_1,t_2,\ldots$.
Put $S_i:=\beta(t_i)\cap\beta(t_{i+1})$.
Finite adhesion gives an integer $a$ such that $|S_i|\leq a$ for infinitely many $i$.
Choose $a+1$ pairwise disjoint rays representing $\omega$, and fix one vertex on each of them and a bag containing each chosen vertex.
For all sufficiently large $i$, the chosen bags lie on the $t_i$-side of $t_it_{i+1}$, while tails of the rays lie on the strict $t_{i+1}$-side.
Each ray therefore meets $S_i$.
Choosing such an $i$ with $|S_i|\leq a$ contradicts disjointness of the $a+1$ rays.
\end{proof}

We recall the accessibility theorem for quasi-planar graphs.

\begin{theorem}[{\cite[Thm.~A]{macmanus2023accessibility}}]\label{thm:quasi-planar-accessible}
Every connected, locally finite, quasi-transitive, quasi-planar graph is accessible.
\end{theorem}

We will also need the following characterisation of quasi-trees. 

\begin{theorem}[{\cite[Thm.~5.5]{KM08}}]\label{lem:thin-ended-almost-planar}
A connected, locally finite, quasi-transitive graph with no thick end is a quasi-tree.
\end{theorem}

We use the following standard consequence of the structure-tree construction of Dicks--Dunwoody \cite[Thm.~II.2.20]{dicks1989groups} and Thomassen--Woess \cite[Sec.~7]{thomassen1993vertex}; see also \cite[Sec.~4]{hamann2022stallings} and \cite[Thm.~2.15]{macmanus2023accessibility}.

\begin{proposition}\label{prop:highly-connected-parts}
Let $\Gamma$ be a connected, locally finite, quasi-transitive graph.
For every integer $n\geq1$, there is a canonical tree-decomposition $(T,\beta)$ of $\Gamma$ with the following properties:
\begin{enumerate}
\item The parts are connected, adhesion is bounded, and $T/\operatorname{Aut}(\Gamma)$ is finite.
\item The stabiliser of $t$ acts quasi-transitively on the part $P_t:=\Gamma[\beta(t)]$.
\item Every part satisfies $\operatorname{es}(P_t)\geq n$.
\item The inclusion $P_t\hookrightarrow\Gamma$ is a quasi-isometric embedding and induces an injection $\Omega(P_t)\hookrightarrow\Omega(\Gamma)$.
\end{enumerate}
\end{proposition}

If $\Gamma$ is accessible, the same construction gives a canonical tree-decomposition with bounded adhesion, connected finite or one-ended parts, and finite quotient; see \cite[Sec.~7]{thomassen1993vertex} and \cite[Thm.~6.4, Prop.~4.1]{hamann2022stallings}.

\subsection{Near-embeddings and minor exclusion}\label{subsec:near-embeddings}

We now describe the conventions surrounding `near-embeddings' into compact surfaces, and their link to the structure theory of minor-excluded graphs.

Let $\Sigma$ be a compact connected surface with boundary cycles $C_1,\ldots,C_k$.
We will always fix, for each $1 \leq i \leq k$, a continuous parametrisation $c_i:[0,1]\to C_i$ which is injective except that $c_i(0)=c_i(1)$.
We call this common endpoint the \emph{root} of $C_i$. This induces a total order on $C_i\setminus\{c_i(0)\}$.
By an \emph{embedding} of a graph in $\Sigma$, we will mean a continuous injective map from its geometric realisation as a CW-complex into $\Sigma$ whose image meets $\partial\Sigma$ only in vertices and avoids the roots of the boundary cycles.

\begin{definition}[Near-embedding]\label{def:near-embedding}
Let $\Gamma$ be a graph and let $\Sigma$ be a compact connected surface with boundary cycles $C_1,\ldots,C_k$.
A \emph{near-embedding} of $\Gamma$ in $\Sigma$ consists of the following data:
\begin{enumerate}
\item A set $X\subseteq V(\Gamma)$ with $|X|\leq k$, called the \emph{deleted set}.
\item Subgraphs $\Gamma_0,\Gamma_1,\ldots,\Gamma_k\subseteq\Gamma-X$ whose union is $\Gamma-X$, with $\Gamma_1,\ldots,\Gamma_k$ pairwise vertex-disjoint.
\item An embedding of $\Gamma_0$ in $\Sigma$ such that, for each $i$, the set $U_i:=V(\Gamma_0)\cap V(\Gamma_i)$ is precisely $V(\Gamma_0)\cap C_i$, identifying $\Gamma_0$ with its image.
\item For each $i$, a linear decomposition $(U_i,\beta_i)$ of $\Gamma_i$ of width less than $k$, using the order on $U_i$ induced by $C_i$, and satisfying $u\in\beta_i(u)$ for every $u\in U_i$.
\end{enumerate}
The graph $\Gamma_0$ is the \emph{embedded part}, the graphs $\Gamma_1,\ldots,\Gamma_k$ are the \emph{vortices}, and the sets $U_1,\ldots,U_k$ are their \emph{societies}.
\end{definition}

This follows the convention set out by Diestel--Thomas in \cite{diestel1999excluding}. 
The use of the same integer $k$ to bound the number of boundary cycles, the size of the deleted set, and the widths of the vortices is purely a bookkeeping convention.
When the deleted set is empty, this is the infinite-graph analogue of the notion of an `almost-embedding' used by Grohe in  \cite{Grohe03}.

\begin{remark}
    We highlight that, under the conventions laid out in Definition~\ref{def:near-embedding}, the property of admitting a near-embedding into a fixed compact surface $\Sigma$ does not pass to subgraphs na\"ively. In particular, the subgraph in question may exclude some key vertices of the societies. It is, however, straightforward to `enlarge' our subgraph so that it inherits a near-embedding, and we will employ this trick a few times. See, for example,  Lemma~\ref{lem:finite-near-embedded-subgraph} below.
\end{remark}

We will make repeated use of the following basic fact. 

\begin{proposition}[{\cite[Lem.~2.3]{diestel1999excluding}}]\label{prop:near-embeddings-exclude-minors}
    For every compact connected surface $\Sigma$, there exists $N \geq 0$ such that for every graph $\Gamma$ which admits a near-embedding into $\Sigma$, we have $K_N \not\prec \Gamma$.
\end{proposition}

Following \cite{diestel1999excluding}, we introduce the following shorthand. 

\begin{definition}[Plane graph with at most one vortex]\label{def:one-vortex}
A graph is a \emph{plane graph with at most one vortex} if it admits a near-embedding in a surface obtained from the sphere by deleting the interiors of finitely many pairwise disjoint discs, with at most one nonempty vortex.
\end{definition}

This definition requires the deleted set and vortex width to be finite, but places no uniform bound on them over the class of all such graphs.
The following theorem of Diestel--Thomas extends the structure theorem of Robertson--Seymour for finite graphs excluding a particular clique \cite{robertson2003graph}.

\begin{theorem}[{\cite[Thm.~3.2]{diestel1999excluding}}]\label{thm:diestel-thomas}
A graph excludes a countable clique if and only if it admits a tree-decomposition of finite adhesion in which every torso is a plane graph with at most one vortex.
\end{theorem}

Here finite adhesion has the meaning specified in Definition~\ref{def:finite-adhesion}.
For exclusion of a fixed finite clique, the stronger structure theorem gives uniform bounds throughout the decomposition.

\begin{theorem}[{\cite[Thm.~3.1, Lem.~2.3]{diestel1999excluding}}]\label{thm:finite-clique-structure}
    For every integer $N \geq 5$, there is a finite collection $\mathcal S$ of compact connected surfaces such that every countable graph excluding $K_N$ admits a tree-decomposition of bounded adhesion in which every torso admits a near-embedding in some member of $\mathcal S$.
\end{theorem}

We now state Grohe's theorem on local tree-width, which will play an important role in our proofs.

\begin{definition}[Local tree-width]\label{def:local-tree-width}
The \emph{local tree-width function} of a graph $\Gamma$ is defined by 
$$
\operatorname{ltw}_\Gamma(r):=\sup_{v\in V(\Gamma)}\operatorname{tw}(\Gamma[B_\Gamma(v,r)])
$$
for $r\in\mathbb{N}$.
A class $\mathcal{C}$ of graphs has \emph{linear local tree-width} if there is a constant $c$ such that $\operatorname{ltw}_\Gamma(r)\leq cr$ for every $\Gamma\in\mathcal{C}$ and every integer $r\geq1$.
\end{definition}

\begin{theorem}[{\cite[Prop.~3.8]{Grohe03}}]\label{thm:grohe-local-tree-width}
Let $\Sigma$ be a compact connected surface with boundary.
The class of finite graphs admitting a near-embedding in $\Sigma$ with empty deleted set has linear local tree-width.
\end{theorem}

The careful reader might note that this result pertains to finite graphs.
The following lemma will allow us to apply this theorem to the infinite graphs considered in this paper.

\begin{lemma}\label{lem:finite-near-embedded-subgraph}
Let $\Gamma$ admit a near-embedding in a compact connected surface $\Sigma$ with empty deleted set.
Every finite subgraph of $\Gamma$ is contained in a finite subgraph admitting a near-embedding in $\Sigma$ with empty deleted set.
\end{lemma}

\begin{proof}
Let $F\subseteq\Gamma$ be finite, and fix the given near-embedding.
For each vortex $\Gamma_i$, choose a finite set $I_i\subseteq U_i$ containing $V(F)\cap U_i$ and indices of bags witnessing every vertex and edge of $F\cap\Gamma_i$.
Let $W_i:=(V(F)\cap V(\Gamma_i))\cup I_i$, and form a finite graph $\Gamma_i'$ from $F\cap\Gamma_i$ by adding the vertices of $I_i$.
The bags $\beta_i(u)\cap W_i$, indexed by $u\in I_i$, give a linear decomposition of $\Gamma_i'$ of the required width, and each bag contains its index.
Let $\Gamma_0'$ be $F\cap\Gamma_0$ together with the vertices in the sets $I_i$, with its inherited embedding in $\Sigma$.
Then $\Gamma_0'\cup\Gamma_1'\cup\cdots\cup\Gamma_k'$ is a finite subgraph of $\Gamma$ containing $F$, and these data give the required near-embedding with societies $I_i$.
\end{proof}

\section{Escaping the vortices}\label{sec:escaping-vortex}

In this section we consider graphs admitting regular maps or quasi-isometries to (subgraphs of) graphs which themselves admit near-embeddings into compact surfaces. The goal is to settle the case where the Diestel--Thomas tree-decomposition given by Theorem~\ref{thm:diestel-thomas} or \ref{thm:finite-clique-structure} is trivial. More precisely, we will prove the following. 

\begin{theorem}\label{thm:escaping-vortex}
Let $\Gamma$ be a connected, locally finite, quasi-transitive graph, and let $\Pi$ be a graph admitting a near-embedding in a compact connected surface.
\begin{enumerate}
\item If $\Gamma$ admits a regular map into a connected bounded-degree subgraph of $\Pi$, then $\Gamma$ is almost planar.
\item If $\Gamma$ is quasi-isometric to a connected bounded-degree subgraph of $\Pi$, then $\Gamma$ is quasi-planar. 
\end{enumerate}
\end{theorem}

We begin with the following lemma, which applies Grohe's theorem on local tree-width (\ref{thm:grohe-local-tree-width}) to show that, if a subgraph $H$ of a graph $\Pi$  admitting a near-embedding into a compact surface is contained in some finite neighbourhood of the vortices, then $H$ has finite tree-width.

\begin{lemma}\label{lem:dense-vortex-tree-width}
Let $\Pi$ be a graph with a fixed near-embedding in a compact connected surface, with vortices $Q_1,\ldots,Q_q$. Let $H\subseteq\Pi$ be connected and locally finite, and put $W:=V(H)\cap\bigcup_{i=1}^q V(Q_i)$. If $W$ is coarsely dense in $H$, then $H$ has finite tree-width.
\end{lemma}

\begin{proof}
The idea is to cone off the vortices, bringing our graph into a ball of bounded radius to which Grohe's theorem applies. We first suppose that the deleted set is empty, and let $\Sigma$ be the surface of the given near-embedding. Choose $R$ such that $W$ is $R$-coarsely dense in $H$.

Let $F\subseteq H$ be a finite connected subgraph. Join each vertex of $F$ to $W$ by a path in $H$ of length at most $R$. By Lemma~\ref{lem:finite-near-embedded-subgraph}, these paths and $F$ are contained in a finite subgraph $L\subseteq\Pi$ with an inherited near-embedding in $\Sigma$. In particular, the endpoints of our paths retain their vortex membership.

Cone off each nonempty vortex of $L$ by adding a new vertex adjacent to all its vertices, and denote the resulting graph by $\widehat L$. Including each cone vertex in every bag of its vortex increases the width by at most one. We accommodate this by adding one empty cuff to $\Sigma$, choosing a disc disjoint from the finite embedded part. Thus, Theorem~\ref{thm:grohe-local-tree-width} gives a common local tree-width bound $\operatorname{ltw}_{\widehat L}(r) \leq cr$ for all the graphs $\widehat L$, independently of $F$.

Each vertex of $F$ is within distance $R+1$ of one of the cone vertices. For adjacent vertices of $F$, the corresponding cone vertices are therefore at distance at most $2R+3$. Since $F$ is connected and there are at most $q$ cone vertices, these can be linked by chains of at most $q-1$ such steps. Consequently, $F$ lies in a ball of radius at most $q(2R+3)$ in $\widehat L$, giving
$$
\operatorname{tw}(F)\leq cq(2R+3).
$$
This bound is independent of $F$. Since $F$ was an arbitrary finite connected subgraph of $H$, it follows that $\operatorname{tw}(H)$ satisfies the same bound, and in particular is finite. 

Now let $X$ be a possibly nonempty deleted set, and choose $R$ such that $W$ is $R$-coarsely dense in $H$. Every vertex outside the finite set $B_H(X\cap V(H),R)$ has a path of length at most $R$ to $W$ avoiding $X$. It follows that $W\cap V(C)$ is coarsely dense in every infinite component $C$ of $H-X$, with distances measured in $C$. The preceding argument, applied inside $\Pi-X$, therefore gives finite tree-width for each such component. Since $H-X$ has only finitely many components by connectedness and local finiteness of $H$, restoring the finite set $X\cap V(H)$ proves the result.
\end{proof}

Away from the vortices, we must also remove the obstruction coming from the topology of the surface.
For a countable embedded graph, deleting finitely many vertices suffices.

\begin{lemma}\label{lem:finite-planarising-set}
Let $G$ be a countable graph embedded in a compact surface.
There is a finite set $Z\subseteq V(G)$ such that $G-Z$ is planar.
\end{lemma}

\begin{proof}
Apply \cite[Thm.~1.2]{georgakopoulos2025excluded} to decompose $G$ into finitely many planar subgraphs with finite pairwise intersections.
Deleting the union of these intersections leaves a disjoint union of planar graphs.
\end{proof}

The above results tell us that either our quasi-transitive source graph has finite tree-width, or there are arbitrarily large regions of the image which appear planar. We now  set up the machinery required to take advantage of this. 

A \emph{pointed graph} is a pair $(\Lambda,u)$ consisting of a connected graph and a distinguished vertex.
A sequence $(\Lambda_n,u_n)$ \emph{converges} to $(\Lambda_\infty,u_\infty)$ if, for every integer $r\geq0$, the induced balls $\Lambda_n[B_{\Lambda_n}(u_n,r)]$ and $\Lambda_\infty[B_{\Lambda_\infty}(u_\infty,r)]$ are isomorphic as pointed graphs for all sufficiently large $n$.
The following compactness statement follows by a standard diagonal argument. 

\begin{proposition}\label{prop:pointed-compactness}
Every sequence of pointed graphs with uniformly bounded degree has a convergent subsequence. 
\end{proposition}

We next record the following standard construction.

\begin{proposition}\label{prop:pointed-limits-maps}
Let $(\Gamma,v)$ be a pointed graph, and suppose that $(\Pi_n,u_n)$ converges to $(\Pi_\infty,u_\infty)$. Assume that all these graphs have uniformly bounded degree. Let $f_n:\Gamma\to\Pi_n$ satisfy $f_n(v)=u_n$ for every $n$.
\begin{enumerate}
\item If the maps $f_n$ are $\kappa$-regular for a common constant $\kappa$, then there is a $\kappa$-regular map $f_\infty:\Gamma\to\Pi_\infty$.
\item If the maps $f_n$ are $\lambda$-quasi-isometries for a common constant $\lambda$, then there is a $\lambda$-quasi-isometry $f_\infty:\Gamma\to\Pi_\infty$.
\end{enumerate}
In both cases, we may take $f_\infty(v)=u_\infty$.
\end{proposition}

\begin{proof}
After passing to a subsequence, choose pointed isomorphisms $\varphi_n$ between the balls of radius $n$ about $u_n$ and $u_\infty$. For each fixed $x\in V(\Gamma)$, the points $\varphi_n(f_n(x))$ are defined for all sufficiently large $n$ and lie in a fixed finite ball. Since $\Gamma$ is countable, a diagonal argument gives a further subsequence along which these points are eventually constant for every $x$. We then define
$$
f_\infty(x):=\lim_{n\to\infty}\varphi_n(f_n(x)).
$$
It is straightforward to verify that $f_\infty$ has the required properties, and we leave this as an exercise to the reader. 
\end{proof}

We now combine these observations to prove the main result of this section. 

\begin{proof}[Proof of Theorem~\ref{thm:escaping-vortex}]
The conclusions are immediate when $\Gamma$ is finite or two-ended, so assume otherwise.
In the regular-map case, Lemma~\ref{lem:regular-connected-image} gives a connected subgraph $H$ of the given target and a coarsely surjective regular map $f:\Gamma\to H$.
In the quasi-isometry case, we take $f:\Gamma\to H$ to be a quasi-isometry.
In both cases, $H$ is a connected bounded-degree subgraph of a graph $\Pi$ admitting a near-embedding in a compact surface.
Fix such a near-embedding, with deleted set $X$, embedded part $\Pi_0$, and vortices $Q_1,\ldots,Q_q$, and put $W:=V(H)\cap\bigcup_{i=1}^q V(Q_i)$.
By Lemma~\ref{lem:finite-planarising-set}, there is a finite set $Z\subseteq V(H\cap\Pi_0)$ such that $(H\cap\Pi_0)-Z$ is planar.

Suppose first that $W$ is coarsely dense in $H$.
By Lemma~\ref{lem:dense-vortex-tree-width}, $H$ has finite tree-width, so every end of $H$ is thin.
Lemma~\ref{lem:regular-map-on-ends} implies that $\Gamma$ has no thick end.
Theorem~\ref{lem:thin-ended-almost-planar} then gives both conclusions.

We may therefore assume that $W$ is not coarsely dense in $H$.
Note that the set $W\cup((X\cup Z)\cap V(H))$ is also not coarsely dense in $H$, since $X$ and $Z$ are finite. 
Since $f$ is coarsely surjective, we may choose $x_n\in V(\Gamma)$ such that, writing $u_n:=f(x_n)$, we have
\begin{equation}\label{eq:escaping-vortex-basepoints}
d_H\bigl(u_n,W\cup((X\cup Z)\cap V(H))\bigr)\longrightarrow\infty.
\end{equation}
Quasi-transitivity allows us to pass to a subsequence for which all $x_n$ belong to one $\operatorname{Aut}(\Gamma)$-orbit.
Choose $v$ in this orbit and automorphisms $g_n$ satisfying $g_n(v)=x_n$.
The maps $f_n:=f\circ g_n$ are regular with common constants and satisfy $f_n(v)=u_n$.
In the quasi-isometry case they are also quasi-isometries with a common constant.
By Proposition~\ref{prop:pointed-compactness}, after passing to a further subsequence, $(H,u_n)$ converges to a pointed graph $(H_\infty,u_\infty)$ with the same degree bound as $H$.
Proposition~\ref{prop:pointed-limits-maps} gives a regular map $\Gamma\to H_\infty$, and gives a quasi-isometry in the quasi-isometry case.

It remains to show that $H_\infty$ is planar. 
Fix an integer $r\geq0$.
By \eqref{eq:escaping-vortex-basepoints}, for all sufficiently large $n$, the ball $B_H(u_n,r)$ avoids $X\cup Z\cup\bigcup_{i=1}^q V(Q_i)$.
Its induced subgraph is therefore contained in the planar graph $(H\cap\Pi_0)-Z$.
Pointed convergence implies that the ball of radius $r$ about $u_\infty$ also induces a planar graph.
Every finite subgraph of $H_\infty$ lies in such a ball, so compactness of planarity implies that $H_\infty$ is planar.
Theorem~\ref{thm:almost-planar-regular} now gives almost planarity of $\Gamma$ in the regular-map case, and in the quasi-isometry case we have $\Gamma\simeq_{\mathrm{QI}}H_\infty$. The theorem follows. 
\end{proof}

\section{Localising to a torso}\label{sec:localising-torso}

We now begin to consider the case where the Diestel--Thomas decompositions are non-trivial. In this section we prove a technical result which allows us to `localise' parts of our quasi-transitive graph to specific torsos of this decomposition, effectively breaking down our target graph into pieces to which we may apply Theorem~\ref{thm:escaping-vortex}. This is achieved in two steps, and we begin with the following lemma. 

\begin{lemma}\label{lem:torso-localisation}
Let $\Pi$ be a connected bounded-degree graph with a tree-decomposition $(T,\beta)$, and let $t\in V(T)$. Suppose that there is a constant $a<\infty$ such that $|\beta(t)\cap\beta(s)|\leq a$ for every neighbour $s$ of $t$. Let $H\subseteq\Pi$ be an infinite connected subgraph satisfying $(\ddagger)$, and suppose that every component of $H-\beta(t)$ is finite. Then $H$ is quasi-isometric to a connected bounded-degree subgraph $M$ of the torso $\Pi[\![\beta(t)]\!]$.
\end{lemma}

\begin{proof}
The idea is to replace the components outside the bag by cliques on their boundaries. Put $Z:=V(H)\cap\beta(t)$, which is nonempty since $H$ is infinite.
For every component $C$ of $H-Z$, let 
$$
\partial C := \{v \in Z : d_H(v, C) = 1\}. 
$$
By the tree-decomposition axioms, $\partial C$ is contained in an adhesion set incident to $t$, so $|\partial C|\leq a$ and $\partial C$ spans a clique in the torso. Since $C$ is a finite component of $H-\partial C$, property $(\ddagger)$ gives a uniform bound $m\geq1$ on the sizes of these components.

Let $M$ be the graph obtained from $H[Z]$ by making each set $\partial C$ a clique. Then $M\subseteq\Pi[\![\beta(t)]\!]$. Replacing excursions outside $Z$ by walks of length at most one in $M$ shows that $M$ is connected and that $d_M(x,y)\leq d_H(x,y)$ for all $x,y\in Z$. Conversely, every added edge can be replaced by a path through the corresponding component $C$ of length at most $m+1$. Thus
$$
d_M(x,y)\leq d_H(x,y)\leq(m+1)d_M(x,y).
$$
Every vertex of $H$ lies within distance $m$ of $Z$, so the identity on $Z$ defines a quasi-isometry from $M$ to $H$. 
Finally, every neighbour of $x$ in $M$ lies in $B_H(x,m+1)$, so $M$ has bounded degree.
\end{proof}

We now apply this lemma together with Theorem~\ref{thm:escaping-vortex} to prove the following, which is the main result of this section. 

\begin{proposition}\label{prop:end-localisation}
Let $\Gamma$ be a connected, locally finite, quasi-transitive graph, and let $f:\Gamma\to\Pi$ be a regular map into a connected bounded-degree graph with a tree-decomposition $(T,\beta)$ of finite adhesion.
Suppose that all ends in $f_*(\Omega(\Gamma))$ are carried by a node $t$, that the adhesion sets incident to $t$ have uniformly bounded cardinality, and that the torso at $t$ admits a near-embedding in a compact connected surface.
Then $\Gamma$ is almost planar.
If $f$ is a quasi-isometric embedding, then $\Gamma$ is quasi-planar.
\end{proposition}

\begin{proof}
If $\Gamma$ has no thick end, both conclusions follow from Theorem~\ref{lem:thin-ended-almost-planar}.
Otherwise, $\Gamma$ is infinite and not two-ended, so it satisfies $(\ddagger)$.
Lemma~\ref{lem:regular-connected-image} gives a connected subgraph $Y\subseteq\Pi$ and a coarsely surjective regular map $h:\Gamma\to Y$ whose composition with inclusion is $f$.
By Lemma~\ref{lem:regular-map-on-ends}, every end of $Y$ is the image under $h_*$ of an end of $\Gamma$, so its image in $\Pi$ belongs to $f_*(\Omega(\Gamma))$ and is therefore carried by $t$.
Every component of $Y-\beta(t)$ is finite.
Indeed, a ray in an infinite component would lie in a component of $\Pi-\beta(t)$ supported on a strict side of an edge incident to $t$, and would orient that edge away from $t$.
By Lemma~\ref{lem:regular-ddagger}, $Y$ satisfies $(\ddagger)$.
Lemma~\ref{lem:torso-localisation} therefore gives $Y\simeq_{\mathrm{QI}}M$, where $M$ is a connected bounded-degree subgraph of the torso at $t$. 
Composing $h$ with a quasi-isometry $Y\to M$ gives a regular map $\Gamma\to M$. 
Since the torso admits a near-embedding in a compact connected surface, Theorem~\ref{thm:escaping-vortex} makes $\Gamma$ almost planar.

If $f$ is a quasi-isometric embedding, then $h$ is a quasi-isometry, since it is regular and coarsely surjective and inherits the lower distance bound from $f$. Thus we have that $\Gamma\simeq_{\mathrm{QI}}M$, and Theorem~\ref{thm:escaping-vortex} gives quasi-planarity.
\end{proof}

\section{One-ended and accessible graphs}\label{sec:one-ended-case}

In this section we apply the above results to prove our main theorems in the case that the input quasi-transitive graph is one-ended. We then describe how the one-ended case can be used to settle the accessible case. In particular, we conclude this section by proving Theorem~\ref{thm:reg-ctble} of the introduction. 

The one-ended case is dealt with by the following theorem. 

\begin{theorem}\label{thm:one-ended-rigidity}
Let $\Gamma$ be a connected, locally finite, quasi-transitive, one-ended graph.
\begin{enumerate}
\item If $\Gamma$ admits a regular map into a connected bounded-degree graph which excludes a countable clique, then $\Gamma$ is almost planar.
\item If $\Gamma$ is quasi-isometric to a connected graph which excludes a countable clique, then $\Gamma$ is quasi-isometric to a planar graph.
\end{enumerate}
\end{theorem}

\begin{proof}
Note that the unique end $\omega$ of $\Gamma$ is necessarily thick. 
In the regular-map case, let $H$ be the given bounded-degree target and let $f:\Gamma\to H$ be the given map.
In the quasi-isometry case, Lemma~\ref{lem:bounded-degree-qi-image} gives a connected bounded-degree subgraph $H$ of the target quasi-isometric to $\Gamma$; choose a quasi-isometry $f:\Gamma\to H$.
In either case, $H$ excludes a countable clique, and $f$ is regular.

By Theorem~\ref{thm:diestel-thomas}, $H$ has a tree-decomposition $(T,\beta)$ of finite adhesion whose torsos are plane graphs with at most one vortex.
Lemma~\ref{lem:regular-map-on-ends} makes $f_*(\omega)$ thick, so Lemma~\ref{lem:thick-end-carried} supplies a node $t$ carrying it.
The torso at $t$ excludes a finite clique by Proposition~\ref{prop:near-embeddings-exclude-minors}.
Since every incident adhesion set is a clique in that torso, these sets have uniformly bounded cardinality.
Now, Proposition~\ref{prop:end-localisation} gives the desired conclusion in each case, and we are done.
\end{proof}

We now prove the following elementary lemma, which enables us to `glue together' instances of the one-ended case to deduce the accessible case. 

\begin{lemma}\label{lem:planar-assembly}
Let $\Gamma$ be a connected, locally finite, quasi-transitive graph, and let $(T,\beta)$ be a canonical tree-decomposition with nonempty connected parts, bounded adhesion, and finite quotient $T/\operatorname{Aut}(\Gamma)$.
Then $\Gamma$ is almost planar if and only if every part is almost planar, and $\Gamma$ is quasi-planar if and only if every part is quasi-planar.
\end{lemma}

\begin{proof}
Write $\Gamma_t:=\Gamma[\beta(t)]$. Finiteness of the quotient gives finitely many part types and a uniform bound on the intrinsic diameters of adhesion sets in their incident parts, and the parts are uniformly quasi-isometrically embedded in $\Gamma$. The forward implications follow from Lemma~\ref{lem:bounded-degree-qi-image} and the fact that almost planarity passes to subgraphs.

For the converse, we assume without loss of generality that no induced separation of the tree-decomposition has an empty strict side, equivariantly contracting redundant edges if necessary.  
Note that each vertex $v$ now belongs to uniformly boundedly many bags. Indeed, let $D$ bound the diameters of the adhesion sets in $\Gamma$. Let $T_v$ be the subtree of bags containing $v$. Every adhesion set corresponding to an edge of $T_v$ lies in the finite ball $B:=B_\Gamma(v,D)$. Choose a finite set $F$ containing $B$ and meeting every component of $\Gamma-B$. For each $u \in F$, choose $t_u \in V(T)$ such that $\beta(t_u) \ni u$. Let $T' \subset T$ denote the convex hull of the $t_u$, which is a finite subtree. Every edge of $T_v$ separates two of these chosen bags, since both strict sides contain a vertex of $F$. Thus $T_v \subset T'$ and so $T_v$ is finite, and quasi-transitivity gives a uniform bound independent of $v$.

Choose regular maps $f_t:\Gamma_t\to \Pi_t$ to connected bounded-degree planar graphs, choosing quasi-isometries in the quasi-planar case. Since there are only finitely many part types up to isomorphism, all constants and degree bounds can be chosen uniformly. Consider the disjoint union of the $\Pi_t$. For each edge $st\in E(T)$, choose $x_{st}\in\beta(s)\cap\beta(t)$ and join $f_s(x_{st})$ to $f_t(x_{st})$ by an edge. The resulting graph $\Pi$ is planar, since every joining edge is a bridge, and the uniform bounds on adhesion diameters and bag multiplicities in our tree-decomposition imply that $\Pi$ has bounded degree.
For each $v\in V(\Gamma)$, choose a bag $\beta(t(v))$ containing $v$, and define
$$
f(v):=f_{t(v)}(v).
$$
It is easy to verify that $f$ is regular, and if the $f_t$ are (uniform) quasi-isometries, then so is $f$. We leave these routine estimates to the reader; compare \cite[Sec.~3]{Esperet_Giocanti_2024}. Theorem~\ref{thm:almost-planar-regular} gives almost planarity in the regular case, and we are done.
\end{proof}

We are now able to deduce Theorem~\ref{thm:reg-ctble} of the introduction, which we restate below for the convenience of the reader. 

\regctble*

\begin{proof}
Choose a canonical tree-decomposition of $\Gamma$ with bounded adhesion, connected finite or one-ended parts,  and finite quotient.
The given regular map restricts to a regular map on each part, so every one-ended part is almost planar by Theorem~\ref{thm:one-ended-rigidity}.
Finite parts are trivially almost planar as well, and so Lemma~\ref{lem:planar-assembly} gives the desired conclusion.
\end{proof}

\section{Localising inaccessible parts}\label{sec:accessibility}

The goal of this section is to prove Theorem~\ref{thm:main} of the introduction. The main hurdle to overcome is that we do not yet know our input graph is accessible. The rough idea is that if it were inaccessible, then we could apply the localisation trick of Proposition~\ref{prop:end-localisation} to find a quasi-transitive subgraph which is both inaccessible and quasi-planar. This then contradicts the quasi-planar accessibility theorem (\ref{thm:quasi-planar-accessible}). 

We proceed in several steps.

\begin{lemma}\label{lem:ultra-thick-in-part}
Let $\Gamma$ be a connected, locally finite, quasi-transitive graph with a canonical tree-decomposition $(T,\beta)$ of bounded adhesion.
Suppose that $T/\operatorname{Aut}(\Gamma)$ is finite and that every part is connected and quasi-transitive under its node stabiliser.
If an ultra-thick end $\omega$ is carried by $t\in V(T)$, then it is represented by an ultra-thick end of $P:=\Gamma[\beta(t)]$.
\end{lemma}

\begin{proof}
Let $a \geq 0$ bound the adhesion of $(T,\beta)$. Finiteness of $T/\operatorname{Aut}(\Gamma)$ gives a constant $D\geq1$ bounding the intrinsic diameters in $P$ of the adhesion sets incident to $t$. Every excursion outside $P$ has its endpoints in a common incident adhesion set, so it can be replaced by a path in $P$ of length at most $D$.

Every ray representing an end carried by $t$ meets $P$ infinitely often. Replacing its excursions gives a proper walk in $P$ representing the same end, and deleting loops gives a ray. In particular, choose an end $\omega_P$ of $P$ representing $\omega$. If a finite set $S\subseteq V(P)$ separates two ends of $P$, then $B_P(S,D)$ separates their images in $\Gamma$, since a path avoiding this neighbourhood could have all its excursions replaced by paths avoiding $S$.

Suppose that $\omega_P$ is not ultra-thick, and choose $k$ such that it can be separated from every other end of $P$ by at most $k$ vertices. The preceding observation shows that $\omega$ can be separated from every other end of $\Gamma$ carried by $t$ by at most $k v_P(D)$ vertices. Every end not carried by $t$ is separated from $\omega$ by an incident adhesion set, of size at most $a$. Thus $\omega$ is not ultra-thick, a contradiction.
\end{proof}

\begin{lemma}\label{prop:inaccessible-almost-planar-part}
Let $\Gamma$ be an inaccessible, connected, locally finite, quasi-transitive graph, and let $f:\Gamma\to H$ be a regular map into a connected bounded-degree graph excluding a countable clique.
Then $\Gamma$ has a connected quasi-transitive subgraph $P$ which is inaccessible and almost planar.
If $f$ is a quasi-isometry, then $P$ is quasi-planar.
\end{lemma}

\begin{proof}
Choose an ultra-thick end $\omega\in\Omega(\Gamma)$.
By Lemma~\ref{lem:regular-map-on-ends}, the image end $\bar\omega:=f_*(\omega)$ is thick.
Theorem~\ref{thm:diestel-thomas} gives a tree-decomposition $(T,\beta)$ of $H$ of finite adhesion in which every torso is a plane graph with at most one vortex.
By Lemma~\ref{lem:thick-end-carried}, $\bar\omega$ is carried by some $t\in V(T)$.
The torso at $t$ excludes a finite clique $K_r$ as a minor by Proposition~\ref{prop:near-embeddings-exclude-minors}.
Since each adhesion set incident to $t$ is a clique in that torso, all such adhesion sets have cardinality less than $r$.
Every end of $H$ not carried by $t$ can therefore be separated from $\bar\omega$ by fewer than $r$ vertices, using an adhesion set incident to $t$.
Lemma~\ref{lem:regular-end-separators} supplies $b<\infty$ such that, for every $\xi\in\Omega(\Gamma)$,
\begin{equation}\label{eq:ends-outside-target-torso}
f_*(\xi)\text{ is not carried by }t\quad\Longrightarrow\quad\operatorname{vs}(\omega,\xi)\leq b.
\end{equation}

Write $\Delta:=\Delta(\Gamma)$ and choose an integer $n>\Delta b$.
Apply Proposition~\ref{prop:highly-connected-parts} with this value of $n$ to obtain a canonical tree-decomposition $(S,\alpha)$ of $\Gamma$.
By Lemma~\ref{lem:thick-end-carried}, the end $\omega$ is carried by some $u\in V(S)$.
Put $P:=\Gamma[\alpha(u)]$.
Lemma~\ref{lem:ultra-thick-in-part} gives an ultra-thick end $\omega_P$ of $P$ representing $\omega$.
Thus $P$ is inaccessible, and in particular is infinite and not two-ended.
By the properties of the decomposition,
\begin{equation}\label{eq:high-connectivity-part}
\operatorname{vs}(P)\geq\Delta^{-1}\operatorname{es}(P)\geq n/\Delta>b.
\end{equation}
Every end of $P$, viewed as an end of $\Gamma$, maps to an end of $H$ carried by $t$.
For otherwise, \eqref{eq:ends-outside-target-torso} would give a set $A\subseteq V(\Gamma)$ with $|A|\leq b$ separating that end from $\omega$.
Its intersection with $V(P)$ would separate the corresponding end of $P$ from $\omega_P$, contrary to \eqref{eq:high-connectivity-part}.

Proposition~\ref{prop:end-localisation}, applied to $f|_P:P\to H$ and the node $t$, now makes $P$ almost planar.
If $f$ is a quasi-isometry, its restriction $f|_P$ is a quasi-isometric embedding by Proposition~\ref{prop:highly-connected-parts}, so Proposition~\ref{prop:end-localisation} makes $P$ quasi-planar.
\end{proof}

We now have all we need to complete the proof of Theorem~\ref{thm:main} of the introduction, which we restate below for convenience. 

\main*

\begin{proof}
By Lemma~\ref{lem:bounded-degree-qi-image}, there is a connected bounded-degree graph $H$ excluding a countable clique and a quasi-isometry $f:\Gamma\to H$.
If $\Gamma$ were inaccessible, Lemma~\ref{prop:inaccessible-almost-planar-part} would give an inaccessible quasi-transitive part quasi-isometric to a planar graph, contrary to Theorem~\ref{thm:quasi-planar-accessible}.
Thus $\Gamma$ is accessible, and we may choose a canonical tree-decomposition with bounded adhesion, connected finite or one-ended parts, and finite quotient.
Every part $P$ is quasi-isometrically embedded in $\Gamma$, so Lemma~\ref{lem:bounded-degree-qi-image}, applied to $f|_P$, gives a connected subgraph of $H$ quasi-isometric to $P$.
Theorem~\ref{thm:one-ended-rigidity} therefore makes every one-ended part quasi-planar. Finite parts are trivially quasi-planar as well, and so Lemma~\ref{lem:planar-assembly} completes the proof.
\end{proof}

\section{Excluding a finite clique}\label{sec:finite-clique}

We conclude this paper by proving Theorem~\ref{thm:reg-finite} of the introduction. 

\regfinite*

\begin{proof}
Let $f:\Gamma\to H$ be the given map, and choose $\kappa$ such that $f$ is $\kappa$-regular.
Choose $N\geq5$ such that $K_N\nprec H$.
By Theorem~\ref{thm:finite-clique-structure}, there is a tree-decomposition $(T,\beta)$ of $H$ whose adhesion is bounded by an integer $a\geq1$ and whose torsos admit near-embeddings in compact surfaces from a fixed finite collection.
Put $\Delta:=\max\{1,\Delta(\Gamma)\}$ and $b:=\kappa v_H(\kappa)a$, and choose an integer $n>\Delta b$.
Apply Proposition~\ref{prop:highly-connected-parts} to obtain a canonical tree-decomposition $(S,\alpha)$ of $\Gamma$ with connected quasi-transitive parts, bounded adhesion, and finite quotient such that every part $P:=\Gamma[\alpha(u)]$ satisfies
\begin{equation}\label{eq:finite-clique-source-threshold}
\operatorname{vs}(P)\geq\Delta^{-1}\operatorname{es}(P)\geq n/\Delta>b.
\end{equation}
We show that every such part is almost planar.

If $P$ has no thick end, Theorem~\ref{lem:thin-ended-almost-planar} applies.
Otherwise, choose a thick end $\omega\in\Omega(P)$.
The restriction $f_P:P\to H$ is $\kappa$-regular, because inclusion $P\hookrightarrow\Gamma$ is injective and $1$-Lipschitz.
By Lemmas~\ref{lem:regular-map-on-ends} and~\ref{lem:thick-end-carried}, the image end $(f_P)_*(\omega)$ is thick and is carried by a node $t\in V(T)$.
Every end of $P$ maps to an end carried by $t$.
Indeed, if the image of an end $\xi\in\Omega(P)$ were not carried by $t$, an adhesion set $A$ on an edge incident to $t$ would separate its image from $(f_P)_*(\omega)$.
The pullback $f_P^{-1}(B_H(A,\kappa))$ would separate $\xi$ from $\omega$ and have cardinality at most $\kappa v_H(\kappa)|A|\leq b$, contradicting \eqref{eq:finite-clique-source-threshold}.

The torso at $t$ admits a near-embedding in a compact surface and its incident adhesion sets have size at most $a$.
Proposition~\ref{prop:end-localisation}, applied to $f_P$ and $t$, therefore makes $P$ almost planar.
All parts of $(S,\alpha)$ are now almost planar.
Lemma~\ref{lem:planar-assembly} implies that $\Gamma$ is almost planar, as required.
\end{proof}

\bibliographystyle{abbrv}
\bibliography{references}

\end{document}